\documentclass[12pt]{amsart}
\usepackage{amssymb}
\newtheorem{theorem}{Theorem}
\newtheorem{lemma}[theorem]{Lemma}
\newtheorem{proposition}[theorem]{Proposition}
\newtheorem{corollary}[theorem]{Corollary}
\newtheorem{example}[theorem]{Example}
\newtheorem{question}[theorem]{Question}

\begin{document}

\title[Upper bounds for the tightness]{Upper bounds for the tightness of the $G_\delta$-topology}

\author[A. Bella]{Angelo Bella}
\address{Department of Mathematics and Computer Science \\
University of Catania \\
viale A. Doria 6 \\
95125 Catania, Italy}
\email{bella@dmi.unict.it}

\author[S. Spadaro]{Santi Spadaro}
\address{Department of Mathematics and Computer Science \\
University of Catania \\
viale A. Doria 6 \\
95125 Catania, Italy}
\email{santidspadaro@gmail.com}


\begin{abstract}
We prove that if $X$ is a regular space with no uncountable free sequences, then the tightness of its $G_\delta$ topology is at most continuum and if $X$ is in addition Lindel\"of then its $G_\delta$ topology contains no free sequences of length larger then the continuum. We also show that the higher cardinal generalization of our theorem does not hold, by constructing a regular space with no free sequences of length larger than $\omega_1$, but whose $G_\delta$ topology can have arbitrarily large tightness.
\end{abstract}

\maketitle 
\bigskip  

\section{Introduction}

Given a space $X$, the \emph{$G_\delta$-modification of $X$} (or $G_\delta$-topology on $X$), $X_\delta$ is defined as the topology on $X$ which is generated by the $G_\delta$-subsets of $X$. The problem of bounding the cardinal invariants of $X_\delta$ in terms of those of $X$ is a well-studied one in set-theoretic topology. For example if $c$, $s$, $L$, $t$ denote respectively the cellularity, the spread, the Lindel\"of degree  and the tightness of $X$, then $c(X_\delta) \leq 2^{c(X)}$ for every compact space $X$ (see \cite{JArhan}), $s(X_\delta) \leq 2^{s(X)}$ for every Hausdorff space $X$ (see \cite{BS}) and $L(X_\delta) \leq 2^{L(X) \cdot t(X)}$ for every Hausdorff space $X$ (see \cite{Pytkeev}). This is nothing but a small sample of bounds for the $G_\delta$ topology that have been proved in the past; for more results and applications of the $G_\delta$ topology to homogeneous compacta we refer the reader to our paper \cite{BS} and its bibliography.

Note that we have not mentioned a bound for the tightness of the $G_\delta$ topology yet, and indeed finding such a bound seems to be particularly tricky. Answering a question posed in \cite{BS}, Dow, Juh\'asz, Soukup, Szentmikl\'ossy and Weiss \cite{DJSSW} proved that  the inequality $t(X_\delta)\le 2^{t(X)}$ holds within the realm of regular Lindel\"of spaces. The Lindel\"of property is essential in their argument, and in fact the authors were able to construct a consistent example of a regular countably tight space $X$ such that $t(X_\delta)$ can be as big as desired. They left open whether a countably tight space $X$ such that $t(X_\delta)>2^{\aleph_0}$ can be found in ZFC. This question was later solved in the positive by Usuba \cite{U}, who also found a bound on the tightness of the $G_\delta$-modification of every countably tight space, modulo the consistency of a certain very large cardinal. More precisely, Usuba proved that if $\kappa$ is an $\omega_1$-strongly compact cardinal then $t(X_\delta) \leq \kappa$, for every countably tight space $X$. Chen and Szeptycki \cite{CS} managed to prove a very tight consistent bound for the special class of Fr\'echet $\alpha_1$-spaces, namely $t(X_\delta) \leq \aleph_1$ if the Proper Forcing Axiom holds. 

Exploiting the notion of a free sequence, we will give another bound on the tightness of the $G_\delta$ topology.

A free sequences is a special kind of discrete set that was introduced by Arhangel'skii and is one of the essential tools in his celebrated solution of the Alexandroff-Urysohn problem on the cardinality of first-countable compacta. Recall that the set $\{x_\alpha: \alpha <\kappa \}\subseteq X$ is a free sequence provided that $\overline {\{x_\beta:\beta<\alpha \}}\cap \overline{\{x_\beta:\alpha \le \beta<\kappa \}}=\emptyset $ for each $\alpha <\kappa $. We define $F(X)$ to be the supremum of cardinalities of free-sequences in $X$. The cardinal functions $F(X)$ and $t(X)$ are intimately related. Indeed, $F(X)\le L(X)t(X)$ for every space $X$ and $t(X)=F(X)$, for every compact Hausdorff space $X$. However, the gap between $F(X)$ and $t(X)$ can be arbitrarily large even for a Lindel\"of space $X$, as observed by Okunev \cite{Ok}.

We will prove a result about the tightness of the $G_\delta$ modification which has the Dow, Juh\'asz, Soukup, Szentmikl\'ossy and Weiss bound as a consequence and also implies the following new bound: if $X$ is a regular space such that $F(X)=\omega$, then $t(X_\delta) \leq 2^{\aleph_0}$. The higher cardinal generalization of this is not true, as we will construct, for every cardinal $\kappa$, a regular space $X$ such that $F(X)=\omega_1 < \kappa=t(X_\delta)$. As a byproduct of our bound we will obtain that if $X$ is a Lindel\"of regular space such that $F(X)=\omega$ then $F(X_\delta) \leq 2^{\aleph_0}$.

Given a set $S$, we denote by $\mathcal{P}(S)$ the powerset of $S$ and by $[S]^{\leq \kappa}$ the set of all subsets of $S$ which have cardinality at most $\kappa$. For undefined notions see \cite{E}, but our notation regarding cardinal functions follows \cite{J}.

 \section{The tightness of the $G_\delta$-modification}

Let $X$ be a space, let $W$ be a subset of $X$ and let $\kappa$ be an infinite cardinal. We say that a collection $\mathcal{U}$ of subsets of $X$ is a $Cl_\kappa $-cover of $W$ provided that  for any  $C \in [W]^{\le \kappa }$ there is  $U_C \in \mathcal{U}$ such that $\overline{C} \subseteq U_C$.

We say that a space $X$ is $Cl_\kappa$-Lindel\"of  if whenever $W$ is a subset of $X$ and $\mathcal U$ is an open $Cl_\kappa $-cover of
$W$, then $W$ is covered by countably many elements of $\mathcal{U}$.

\begin{lemma}  Every Lindel\"of space $X$  is $Cl_{t(X)}$-Lindel\"of. 
\end{lemma}

\begin{proof} It suffices to observe that every open $Cl_{t(X)}$-cover of a set $W \subseteq X$ is actually a cover of $\overline{W}$. 
\end{proof}

\begin{lemma} 
Every space $X$ satisfying $F(X)=\omega$ is $Cl_\omega$-Lindel\"of. 
\end{lemma}

\begin {proof} Let $W$be a subset of $X$ and $\mathcal{U}$ be an open $Cl_\omega$-cover of $W$. Assume by contradiction that no countable subfamily of $\mathcal{U}$ covers $W$. 
We will then construct a free sequence of cardinality $\omega_1$ inside $W$. 

Suppose that, for some $\beta < \omega_1$, we have chosen points $\{x_\tau: \tau < \beta \} \subset W$ and elements $U_\tau\in \mathcal{U}$ for every $\tau < \beta$ with the property that $\overline{\{x_\gamma:\gamma<\tau  \}} \subseteq U_{\tau}$. Choose $U_\beta \in \mathcal{U}$ in such a way that $\overline{\{x_\alpha: \alpha < \beta\}} \subseteq U_\beta$. By
our assumption, the family $\{U_\tau: \tau \leq \beta \}$ does not cover $W$, and therefore we can fix a point $x_\beta \in W \setminus \bigcup \{U_\tau: \tau \leq \beta\}$. 

Eventually, $\{x_\tau: \tau < \omega_1\}$ is a free sequence of cardinality $\omega_1$ in $X$, which is a contradiction.
\end{proof}
          
\begin{theorem} \label{mainthm}
Let $X$ be a regular space and let $\kappa$ be an infinite cardinal. If $X$ is $Cl_\kappa$-Lindel\"of,  then $t(X_\delta) \leq 2^\kappa$.
\end{theorem}

\begin{proof}
Let $A$ be any subset of $X$ and fix a point $p$ in the $G_\delta$-closure of $A$. 

Let $\mathcal{N}_\kappa(X)=\{C \in [X]^{\leq \kappa}: p \notin \overline{C}\}$. By the regularity of $X$, for every $C \in \mathcal{N}_\kappa(X)$, we can find disjoint open sets $U_C$ and $V_C$ such that $\overline{C} \subset U_C$ and $p \in V_C$.

Let $\phi$ be a choice function on $\mathcal{P}(X)$. We will build by induction an increasing family $\{W_\alpha: \alpha < \kappa^+\} \subset [A]^{2^{\kappa}}$.

Let $W_0$ be any subset of $A$ of cardinality $\leq 2^\kappa$ and assume we have already defined $\{W_\beta: \beta < \alpha \}$. If $\alpha$ is a limit ordinal then put $W_\alpha=\bigcup \{W_\beta: \beta < \alpha \}$. If $\alpha=\gamma+1$ then let:

$$W_\alpha=W_\gamma \cup \{\phi(A \cap \bigcap \{V_C: C \in \mathcal{C}\}): \mathcal{C} \in [\mathcal{N}_\kappa(W_\gamma)]^{\leq \omega} \}$$

Note that $|W_\alpha| \leq 2^\kappa$.

Finally, let $W=\bigcup \{W_\alpha: \alpha < \kappa^+\}$. Since $|W| \leq 2^\kappa$, it suffices to show that $p$ is in the $G_\delta$-closure of $W$. 

Indeed, let $\{O_ n:  n<\omega  \}$ be a family of open neighbourhoods of $p$. 

\medskip

\noindent {\bf Claim.} There is a countable family $\mathcal{C}_n \subset \mathcal{N}_\kappa(W \setminus O_n )$ such that $W \setminus O_n \subset \bigcup \{U_C: C \in
\mathcal{C}_n \}$.

\begin{proof}[Proof of Claim]
Since $O_n $ is an open neighbourhood of $p$, we have that $\mathcal{N}_\kappa(W \setminus O_n )=[W \setminus O_n ]^{\leq \kappa}$ and therefore $U_C$ is defined for every $C \in [W \setminus O_n ]^{\leq \kappa}$ and $\overline{C} \subset U_C$. In particular, $\mathcal{U}=\{U_C: C \in \mathcal{N}_\kappa(W \setminus O_n) \}$ is a $Cl_\kappa $-open
cover of $W \setminus O_n$. Now, the statement of the claim follows from the fact that $X$ is a $Cl_\kappa $-Lindel\"of space.
\renewcommand{\qedsymbol}{$\triangle$}
\end{proof}

For every $n<\omega $, fix a family $\mathcal{C}_ n$ satisfying the Claim and let $\mathcal{S}=\bigcup \{\mathcal{C}_ n:n<\omega \}$ and $S=\bigcup \mathcal{S}$. 

Since the set $S$ has cardinality at most $\kappa$, there is an ordinal $\delta < \kappa^+$ such that $S \subset W_\delta$. It follows then that the point $q=\phi(\bigcap_{C \in \mathcal{S}} V_C \cap A)$ belongs to $W_{\delta+1} \subset W$.

Note that, for every $n<\omega $, $q \in \bigcap \{V_C: C \in \mathcal{C}_ n\} \cap W \subset W \setminus \bigcup \{U_C: C \in
\mathcal{C}_n \} \subset O_n \cap W$. Therefore $q \in \bigcap \{O_n :n<\omega  \} \cap W$ and we are done.
\end{proof}

\begin{corollary} (Dow, Juh\'asz, Soukup, Szentmikl\'ossy and Weiss \cite{DJSSW})  
If $X$ is a Lindel\"of regular space, then $t(X_\delta)\le 2^{t(X)}$. 
\end{corollary}

\begin{corollary} \label{cor} 
If $X$ is a regular space and $F(X)=\omega$, then $t(X_\delta)\le 2^\omega$. 
\end{corollary}

It's natural to ask whether the higher cardinal version of Corollary $\ref{cor}$ holds true. The following theorem shows that this is not the case.

Let $\theta$ be a regular uncountable cardinal. Recall that an elementary submodel $M$ of $H(\theta)$ is said to be \emph{$\omega$-covering} if for every countable subset $A$ of $M$ there is a countable set $B \in M$ such that $A \subset B$. The union of any elementary chain of elementary submodels of length $\omega_1$ is an $\omega$-covering elementary submodel, so $\omega$-covering submodels of cardinality $\omega_1$ exist in ZFC (see \cite{D}).

\begin{theorem} \label{example}
For every uncountable cardinal $\kappa$, there is a space $Y$ such that $F(Y)=\omega_1 < \kappa=t(Y_\delta)$.
\end{theorem}

\begin{proof}
Let $X=\Sigma(2^{\kappa})=\{x \in 2^{\kappa}: |x^{-1}(1)| \leq \aleph_0 \}$ and let $p \in 2^{\kappa}$ be the point defined by $p(\alpha)=1$, for every $\alpha < \kappa$. We will prove that $Y=X \cup \{p\}$ with the topology inherited from $2^{\kappa}$ is the required example.

The following Claim was proved by the second author in \cite{S} for the case $\kappa=\omega_2$, but the argument works for any uncountable cardinal $\kappa$ without any modifications. We include it for the reader's convenience.

 \medskip

\noindent {\bf Claim}. $L(X) = \aleph_1$.

\begin{proof}[Proof of Claim]
Let $\mathcal{U}$ be an open cover of $X$. Without loss of generality we can assume that for every $U \in \mathcal{U}$, there is a finite partial function $\sigma: \kappa \to 2$ such that $U=\{x \in 2^{\kappa}: \sigma \subset x \}$. The domain of $\sigma$ will then be called the \emph{support of $U$} and we will write $supp(U)=dom(\sigma)$.

Let $\theta$ be a large enough regular cardinal and $M$ be an $\omega$-covering elementary submodel of $H(\theta)$ such that $X, \mathcal{U}, \kappa \in M$ and $|M|=\aleph_1$.

We claim that $\mathcal{U} \cap M$ covers $X$. Indeed, let $x \in X$ be any point and let $A \in M$ be a countable set such that $x^{-1}(1) \cap M \subset A$. 

Let $Z=\{y \in X: (\forall \alpha \in \kappa \setminus A)(y(\alpha)=0) \}$. Then $Z \in M$ and $Z$ is a compact subspace of $X$. So there is a finite subfamily $\mathcal{V} \in M$ of $\mathcal{U}$ such that $Z \subset \bigcup \mathcal{V}$. Since $\mathcal{V}$ is finite, we have $\mathcal{V} \subset M$. It then follows that $\mathcal{U} \cap M$ covers $Z$. 

Let $a$ be the point such that $a(\alpha)=x(\alpha)$ for all $\alpha \in M \cap \kappa$ and $a(\alpha)=0$ for all $\alpha \in \kappa \setminus M$. The fact that $x^{-1}(1) \cap M \subset A$ implies that $a \in Z$ and hence there is $U \in \mathcal{U} \cap M$ such that $a \in U$. Note that $supp(U)$ is a finite element of $M$ and hence $supp(U) \subset M$. But since $x$ and $a$ coincide on $M$ we then have that $x \in U$ as well, as we wanted.

This proves $L(X) \leq \aleph_1$, but we can't have $L(X)=\aleph_0$ because $X$ is countably compact non-compact. Hence $L(X)=\aleph_1$.

\renewcommand{\qedsymbol}{$\triangle$}
\end{proof}

It is well known that $X$ is Fr\'echet-Urysohn and hence $X$ has countable tightness. Since $F(X) \leq L(X) \cdot t(X)$ we have $F(X) \leq \omega_1$, but then also $F(Y) \leq \omega_1$. It's easy to see that $t(p, Y_\delta)=\kappa$.
\end{proof}

In \cite{CPR} Carlson, Porter and Ridderbos proved the following improvement of the Pytkeev inequality $L(X_\delta) \leq 2^{L(X) \cdot t(X)}$ mentioned in the introduction. 

\begin{theorem} \label{CPRTheorem} \cite{CPR} (Theorem 2.7) 
If $X$ is a Hausdorff space, then $L(X_\delta) \leq 2^{L(X)F(X)}$. 
\end{theorem} 

Putting together Corollary $\ref{cor}$ and the above theorem we obtain: 

\begin{corollary} \label{corfree}
Let $X$ be a regular Lindel\"of space such that $F(X)=\omega$. Then $F(X_\delta) \leq 2^{\aleph_0}$.
\end{corollary}

We don't know whether the Lindel\"of property can be removed from the above corollary.

\begin{question} \label{FreeQuestion}
Let $X$ be a regular space satisfying $F(X)=\omega$. Is it true that $F(X_\delta) \leq 2^\omega$?
\end{question}

It's reasonable to conjecture that the higher cardinal version of Corollary $\ref{corfree}$ holds, at least for Lindel\"of spaces.

\begin{question}
Let $X$ be a regular (Lindel\"of) space. Is it true that $F(X_\delta) \leq 2^{F(X)}$?
\end{question}

Note that neither the consistent example from \cite{DJSSW} of a regular countably tight space $X$ such that $t(X_\delta)$ can be arbitrarily large nor the example from Theorem $\ref{example}$ work for the above question since $F(X)=|X|$ for the former and $F(X_\delta) \leq 2^{\aleph_0}$ for the latter.

We finish with two easy bounds for the tightness of the $G_\delta$ topology, by making using of the weight and the spread.

\begin{proposition} \label{prop} 
Let $X$ be a regular space. Then:
\begin{enumerate}
\item \label{density} $t(X_\delta) \leq 2^{d(X)}$.
\item \label{spread} $t(X_\delta) \leq 2^{s(X)}$.
\end{enumerate}
\end{proposition}

\begin{proof}
To prove ($\ref{density}$) recall that $w(X) \leq 2^{d(X)}$ for every regular space $X$ (see \cite{J}). Now $t(X_\delta) \leq w(X_\delta) \leq w(X)^\omega \leq 2^{d(X) \cdot \omega}=2^{d(X)}$.

To prove ($\ref{spread}$) recall that $nw(X) \leq 2^{s(X)}$ for every regular space $X$ (see \cite{J}) and proceed as before.
\end{proof}

Proposition $\ref{prop}$, ($\ref{density}$) is not true for Hausdorff spaces, as the following example shows.

\begin{example}
There is a separable Hausdorff space $X$ such that $t(X_\delta) > 2^{\aleph_0}$.
\end{example}

\begin{proof}
Let $Y$ be the Kat\v{e}tov extension of the integer. That is, if $\mathcal{U}$ is the set of all non-priincipal ultrafilters on $\omega$ then $Y=\omega \cup \mathcal{U}$, every point of $\omega$ is isolated and a basic neighbourhood of $p \in \mathcal{U}$ is a set of the form $\{p\} \cup A \setminus F$, where $A \in p$ and $F$ is finite.

Let $X=Y \cup \{\infty\}$, where $\infty \notin Y$ and declare $V \subset X$ to be a neighbourhood of $\infty$ if and only if $|X \setminus V| \leq 2^{\aleph_0}$. It is easy to see that $X$ is a separable Hausdorff space and $t(X_\delta) > 2^{\aleph_0}$.
\end{proof}

\section{Acknowledgements}

The authors are grateful to INdAM-GNSAGA for partial financial support and to Lajos Soukup for pointing out an error in a previous version of the paper.

\end{document}